\documentclass[twoside,leqno,british]
{amsart}
\usepackage{amsmath, amssymb, amsbsy, amsthm}
\usepackage{babel,eucal,url,amssymb,
enumerate,amscd,
}

\usepackage[T1]{fontenc}

\DeclareMathOperator{\Id}{Id} 

\begin{document}

\title[Natural connections with torsion expressed by the metrics]
{Natural connections with torsion expressed by the metric tensors
on almost contact manifolds with B-metric}

\author[M. Manev, M. Ivanova]{Mancho Manev, Miroslava Ivanova}


\frenchspacing

\newcommand{\ie}{i.e. }
\newcommand{\X}{\mathfrak{X}}
\newcommand{\W}{\mathcal{W}}
\newcommand{\F}{\mathcal{F}}
\newcommand{\T}{\mathcal{T}}
\newcommand{\U}{\mathcal{U}}
\newcommand{\LL}{\mathcal{L}}
\newcommand{\TT}{\mathfrak{T}}
\newcommand{\M}{(M,\f,\xi,\eta,g)}
\newcommand{\Lf}{(G,\f,\xi,\eta,g)}
\newcommand{\R}{\mathbb{R}}
\newcommand{\s}{\mathfrak{S}}
\newcommand{\n}{\nabla}
\newcommand{\nn}{\tilde{\nabla}}
\newcommand{\tg}{\tilde{g}}
\newcommand{\f}{\varphi}
\newcommand{\D}{{\rm d}}
\newcommand{\id}{{\rm id}}
\newcommand{\al}{\alpha}
\newcommand{\bt}{\beta}
\newcommand{\gm}{\gamma}
\newcommand{\dt}{\delta}
\newcommand{\lm}{\lambda}
\newcommand{\ta}{\theta}
\newcommand{\om}{\omega}
\newcommand{\ea}{\varepsilon_\alpha}
\newcommand{\eb}{\varepsilon_\beta}
\newcommand{\eg}{\varepsilon_\gamma}
\newcommand{\sx}{\mathop{\mathfrak{S}}\limits_{x,y,z}}
\newcommand{\norm}[1]{\left\Vert#1\right\Vert ^2}
\newcommand{\nf}{\norm{\n \f}}
\newcommand{\Span}{\mathrm{span}}
\newcommand{\grad}{\mathrm{grad}}
\newcommand{\thmref}[1]{The\-o\-rem~\ref{#1}}
\newcommand{\propref}[1]{Pro\-po\-si\-ti\-on~\ref{#1}}
\newcommand{\secref}[1]{\S\ref{#1}}
\newcommand{\lemref}[1]{Lem\-ma~\ref{#1}}
\newcommand{\dfnref}[1]{De\-fi\-ni\-ti\-on~\ref{#1}}
\newcommand{\corref}[1]{Corollary~\ref{#1}}



\numberwithin{equation}{section}
\newtheorem{thm}{Theorem}[section]
\newtheorem{lem}[thm]{Lemma}
\newtheorem{prop}[thm]{Proposition}
\newtheorem{cor}[thm]{Corollary}

\theoremstyle{definition}
\newtheorem{defn}{Definition}[section]

\hyphenation{Her-mi-ti-an ma-ni-fold ah-ler-ian}




\begin{abstract}
On a main class of the almost contact manifolds with B-metric, it
is described the family of the linear connections preserving the
manifold's structures by 4 parameters. In this family there are
determined the canonical-type connection and the connection with
zero parameters.
\end{abstract}

\keywords{Almost contact manifold; B-metric; natural connection;
canonical connection; torsion tensor.}

\subjclass[2000]{53C05, 53C15, 53C50.}

\maketitle



\section*{Introduction}

The investigations in the differential geometry of the almost
contact manifolds with B-metric is initiated in \cite{GaMiGr}.
These manifolds are the odd-dimensional extension of the almost
complex manifolds with Norden metric and the case of indefinite
metrics corresponding to the almost contact metric manifolds. The
geometry of the considered manifolds is the geometry of the both
of the structures --- the almost contact structure and the
B-metric. There are important the linear connections with respect
to which the structures are parallel. In the general case such
connections, which is called natural connections, are a countless
number. There are interesting those natural connections, which
torsion tensor is expressed by the metric tensors and the
structural 1-forms of the investigated manifold. In our case, this
condition restricts the manifolds to the class of the considered
manifolds having a conformally equivalent structure of the
parallel almost contact B-metric structure with respect to the
Levi-Civita connection.

Such a problem on Riemannian product manifolds is studied in
\cite{Dobr11-1} and \cite{Dobr11-2}.

The present paper\footnote{Partially supported by project
NI11-FMI-004 of the Scientific Research Fund, Paisii Hilendarski
University of Plovdiv, Bulgaria} is organized as follows. In
Sec.~1 we give some necessary facts about the considered
manifolds.
In Sec.~2 we express the set of the torsion tensors of the natural
connections in terms of the metric tensors on an almost contact
manifold with B-metric as a 4-parametric family.
In Sec.~3 we determine the $\f$-canonical connection and the
natural connection with zero parameters in the mentioned family.

\section{Preliminaries}\label{sec:1}

Let $(M,\f,\xi,\eta,g)$ be an almost contact manifold with
B-metric or an \emph{almost contact B-metric manifold}, \ie $M$ is
a $(2n+1)$-dimensional differen\-tia\-ble manifold with an almost
contact structure $(\f,\xi,\eta)$ consisting of an endomorphism
$\f$ of the tangent bundle, a vector field $\xi$, its dual 1-form
$\eta$ as well as $M$ is equipped with a pseudo-Riemannian metric
$g$  of signature $(n,n+1)$, such that the following algebraic
relations are satisfied
\begin{gather}
\f\xi = 0,\quad \f^2 = -\Id + \eta \otimes \xi,\quad
\eta\circ\f=0,\quad \eta(\xi)=1, \label{str}\nonumber\\%
g(x, y ) = - g(\f x, \f y ) + \eta(x)\eta(y) \label{g}
\end{gather}
for arbitrary $x$, $y$ of the algebra $\X(M)$ on the smooth vector
fields on $M$.

Further, $x$, $y$, $z$ will stand for arbitrary elements of
$\X(M)$.

The associated metric $\tilde{g}$ of $g$ on $M$ is defined by
\begin{equation}\label{tg}
\tilde{g}(x,y)=g(x,\f y)\allowbreak+\eta(x)\eta(y).
\end{equation}
Both metrics
$g$ and $\tilde{g}$ are necessarily of signature $(n,n+1)$. The
manifold $(M,\f,\xi,\eta,\tilde{g})$ is also an almost contact
B-metric manifold.

The covariant derivatives of $\f$, $\xi$, $\eta$ with respect to
the Levi-Civita connection $\n$ play a fundamental role in the
differential geometry on the almost contact manifolds.  The
structural tensor $F$ of type (0,3) on $\M$ is defined by
$
F(x,y,z)=g\bigl( \left( \nabla_x
\f \right)y,z\bigr). $ 
It has the following properties:
\begin{equation*}\label{F-prop}
\begin{split}
F(x,y,z)&=F(x,z,y)
=F(x,\f y,\f z)+\eta(y)F(x,\xi,z) +\eta(z)F(x,y,\xi).
\end{split}
\end{equation*}
The relations of $\n\xi$ and $\n\eta$ with $F$ are:
\begin{equation*}\label{Fxieta}
    \left(\n_x\eta\right)y=g\left(\n_x\xi,y\right)=F(x,\f y,\xi).
\end{equation*}

The following 1-forms are associated with $F$:
\begin{equation*}\label{titi}
\begin{array}{c}
\ta(z)=g^{ij}F(e_i,e_j,z),\quad \ta^*(z)=g^{ij}F(e_i,\f e_j,z),\quad
\om(z)=F(\xi,\xi,z),
\end{array}
\end{equation*}
where $g^{ij}$ are the components of the inverse matrix of $g$
with respect to a basis $\left\{e_i;\xi\right\}$
$(i=1,2,\dots,2n)$ of the tangent space $T_pM$ of $M$ at an
arbitrary point $p\in M$. Obviously, the equality $\om(\xi)=0$ and
the relation $\ta^*\circ\f=-\ta\circ\f^2$ are always valid.

A classification of the almost contact manifolds with B-metric
with respect to $F$ is given in \cite{GaMiGr}. This classification
includes eleven basic classes $\F_1$, $\F_2$, $\dots$, $\F_{11}$.
Their intersection is the special class $\F_0$ determined by the
condition $F(x,y,z)=0$. Hence $\F_0$ is the class of almost
contact B-metric manifolds with $\n$-parallel structures, i.e.
$\n\f=\n\xi=\n\eta=\n g=\n \tilde{g}=0$.

In the present paper we consider the manifolds from the so-called
main classes $\F_1$, $\F_4$, $\F_5$ and $\F_{11}$. These classes
are the only classes where the tensor $F$ is expressed by the
metrics $g$ and $\tilde{g}$. They are defined as follows:
\begin{equation}\label{Fi}
\begin{split}
\F_{1}:\quad &F(x,y,z)=\frac{1}{2n}\bigl\{g(x,\f y)\ta(\f z)+g(\f
x,\f y)\ta(\f^2 z)
\bigr\}_{(y\leftrightarrow z)};\\[4pt]
\F_{4}:\quad &F(x,y,z)=-\frac{\ta(\xi)}{2n}\bigl\{g(\f x,\f y)\eta(z)+g(\f x,\f z)\eta(y)\bigr\};\\[4pt]
\F_{5}:\quad &F(x,y,z)=-\frac{\ta^*(\xi)}{2n}\bigl\{g( x,\f y)\eta(z)+g(x,\f z)\eta(y)\bigr\};\\[4pt]
\F_{11}:\quad
&F(x,y,z)=\eta(x)\left\{\eta(y)\om(z)+\eta(z)\om(y)\right\},
\end{split}
\end{equation}
where (for the sake of brevity) we use the denotation
$\{A(x,y,z)\}_{(y\leftrightarrow z)}$ --- instead of
$\{A(x,y,z)+A(x,z,y)\}$ for any tensor $A(x,y,z)$.

Let us remark that the class
$\F_1\oplus\F_4\oplus\F_5\oplus\F_{11}$, determined by
(\cite{Man-diss})
\begin{equation}\label{F14511}
\begin{split}
    F(x,y,z)=-\frac{1}{2n}\bigl\{&g(\f x,\f y)\ta(z)+g(
x,\f y)\ta^*(z)\\[4pt]
&-2n\,\eta(x)\eta(y)\om(z)\bigr\}_{(y\leftrightarrow z)},
\end{split}
\end{equation}
is the odd-dimensional analogue of the class $\W_1$ of the
conformal K\"ahler manifolds of the corresponding almost complex
manifold with Norden metric, introduced in \cite{GrMeDj}.


\begin{defn}[\cite{Man31}]\label{defn-natural}
A linear connection $D$ is called a \emph{natural connection} on
the manifold  $(M,\f,\allowbreak\xi,\eta,g)$ if the almost contact
structure $(\f,\xi,\eta)$ and the B-metric $g$ are parallel with
respect to $D$, \ie $D\f=D\xi=D\eta=Dg=0$.
\end{defn}
As a corollary, the associated metric $\tilde{g}$ is also parallel
with respect to a natural connection $D$ on $\M$.

According to \cite{ManIv36}, a necessary and sufficient condition
a linear connection $D$ to be natural on $\M$ is $D\f=Dg=0$.

If $T$ is the torsion of $D$, i.e. $T(x,y)=D_x y-D_y x-[x, y]$,
then the corresponding tensor of type (0,3) is determined by
$T(x,y,z)=g(T(x,y),z)$.

Let us denote the difference between the natural connection $D$
and the Levi-Civita connection $\n$ of $g$ by $Q(x,y)=D_xy-\n_xy$
and the corresponding tensor of type (0,3) --- by
$Q(x,y,z)=g\left(Q(x,y),z\right)$.

\begin{prop}[\cite{Man31}]\label{prop-nat-con} A linear connection $D$
is a natural connection on an almost contact B-metric manifold if and only if %
\begin{gather}
 Q(x,y,\f z)-Q(x,\f y,z)=F(x,y,z),
 \nonumber\quad
 Q(x,y,z)=-Q(x,z,y).
 \nonumber
\end{gather}
\end{prop}

Hence and $T(x,y)=Q(x,y)-Q(y,x)$ we have the equality of Hayden's
theorem \cite{Hay}
\begin{equation}\label{Hay}
Q(x,y,z)=\frac{1}{2}\left\{T(x,y,z)-T(y,z,x)+T(z,x,y)\right\}.
\end{equation}

\begin{defn}[\cite{ManIv38}]\label{defn-canonical}
A natural connection $D$ is called a \emph{$\f$-cano\-nic\-al
connection} on the manifold $(M,\f,\xi,\allowbreak\eta,g)$ if the
torsion tensor $T$ of $D$ satisfies the following identity:
\begin{equation}\label{T-can}
\begin{split}
    &\bigl\{T(x,y,z)-T(x,\f y,\f z)
    -\eta(x)\left\{T(\xi,y,z)
    -T(\xi, \f y,\f z)\right\}\\[4pt]
    &-\eta(y)\left\{T(x,\xi,z)-T(x,z,\xi)-\eta(x)T(z,\xi,\xi)\right\}\bigr\}_{[y\leftrightarrow
    z]}=0,
\end{split}
\end{equation}
where we use the denotation $\left\{A(x, y,
z)\right\}_{[y\leftrightarrow z]}$ instead of $A(x, y, z) -
A(x,z,y)$ for any tensor $A(x, y, z)$.
\end{defn}

Let us remark that the restriction the $\f$-canonical connection
$D$ of the manifold $\M$ on the contact distribution $\ker(\eta)$
is the unique canonical connection of the corresponding almost
complex manifold with Norden metric, studied in \cite{GaMi87}.

In \cite{ManGri2}, it is introduced a natural connection on $\M$,
defined by
\begin{equation}\label{fB}
    \n^0_xy=\n_xy+Q^0(x,y),
\end{equation}
where    $Q^0(x,y)=\frac{1}{2}\bigl\{\left(\n_x\f\right)\f
y+\left(\n_x\eta\right)y\cdot\xi\bigr\}-\eta(y)\n_x\xi$.
Therefore, we have
\begin{equation*}\label{Q0}
Q^0(x,y,z)=\frac{1}{2}\left\{F(x,\f y,z)+\eta(z)F(x,\f
y,\xi)-2\eta(y)F(x,\f z,\xi)\right\}.
\end{equation*}
In \cite{ManIv37}, the connection determined by \eqref{fB} is
called a \emph{$\f$B-connection}. It is studied for some classes
of $\M$ in \cite{ManGri2}, \cite{Man3}, \cite{Man4} and
\cite{ManIv37}. The $\f$B-connection is the odd-dimensional
analogue of the B-connection on the corresponding almost complex
manifold with Norden metric, studied for the class $\W_1$ in
\cite{GaGrMi}.

In \cite{ManIv38}, it is proved that the $\f$-canonical connection
and the $\f$B-connection coincide on the almost contact B-metric
manifolds in the class
$\F_1\oplus\F_2\oplus\F_4\oplus\F_5\oplus\F_6\oplus\F_8\oplus\F_9\oplus\F_{10}\oplus\F_{11}$.
Therefore, these connections coincide also in the class
$\F_1\oplus\F_4\oplus\F_5\oplus\F_{11}$.

The torsion of the $\f$-canonical connection has the following
form (\cite{ManIv37})
\begin{equation}\label{T0}
\begin{split}
T^0(x,y,z)=\frac{1}{2}\bigl\{&F(x,\f y,z)+\eta(z)F(x,\f y,\xi)
+2\eta(x)F(y,\f z,\xi)\bigr\}_{[x\leftrightarrow y]}.
\end{split}
\end{equation}

In \cite{ManIv36}, it is given a classification of the linear
connections on the almost contact B-metric manifolds with respect
to their torsion tensors $T$ in 11 classes $\T_{ij}$. The
characteristic conditions of these basic classes are the
following:
\begin{equation}\label{Tij}
\begin{split}
\T_{11/12}:\quad
    &T(\xi,y,z)=T(x,y,\xi)=0,\quad \\[4pt]
    &T(x,y,z)=-T(\f x,\f y,z)=\mp T(x,\f y,\f z);\\[4pt]
%
\T_{13}:\quad &T(\xi,y,z)=T(x,y,\xi)=0,\quad \\[4pt]
            &T(x,y,z)-T(\f x,\f y,z)=\sx T(x,y,z)=0;\\[4pt]
\T_{14}:\quad &T(\xi,y,z)=T(x,y,\xi)=0,\quad \\[4pt]
            &T(x,y,z)-T(\f x,\f y,z)=\sx T(\f x,y,z)=0;\\[4pt]
\T_{21/22}:\quad &T(x,y,z)=\eta(z)T(\f^2 x,\f^2 y,\xi)=\mp \eta(z)T(\f x,\f y,\xi);\\[4pt]
\T_{31/32}:\quad &T(x,y,z)=\eta(x)T(\xi,\f^2 y,\f^2 z)-\eta(y)T(\xi,\f^2 x,\f^2 z),\\[4pt]
                &T(\xi,y,z)=\pm T(\xi,z,y)=-T(\xi,\f y,\f z); \\[4pt]
\T_{33/34}:\quad &T(x,y,z)=\eta(x)T(\xi,\f^2 y,\f^2 z)-\eta(y)T(\xi,\f^2 x,\f^2 z),\\[4pt]
                &T(\xi,y,z)=\pm T(\xi,z,y)=T(\xi,\f y,\f z); \\[4pt]
\T_{41}:\quad
&T(x,y,z)=\eta(z)\left\{\eta(y)\hat{t}(x)-\eta(x)\hat{t}(y)\right\},
\end{split}
\end{equation}
where $\s$ stands for the cyclic sum by three arguments and the
former and latter subscripts of $\T_{ij/kl}$ correspond to upper
and down signs plus or minus, respectively.

According to \cite{ManIv36}, the $\f$B-connection (therefore, in
our case, the $\f$-canonic\-al connection) belongs to the
following class
\begin{equation*}\label{classT0}
\T_{12}\oplus\T_{13}\oplus\T_{14}\oplus\T_{21}\oplus\T_{22}\oplus\T_{31}\oplus\T_{32}\oplus\T_{33}\oplus\T_{34}\oplus\T_{41}.
\end{equation*}

In \cite{ManIv38}, the basic classes of the almost contact
B-metric manifolds are characterized by conditions for the torsion
of the $\f$-canonical connection. For the classes under
consideration we have:
\[
\begin{array}{rl}
\F_1:\; &T^0(x,y)=\frac{1}{2n}\left\{t^0(\f^2 x)\f^2 y-t^0(\f^2 y)\f^2 x
            +t^0(\f x)\f y-t^0(\f y)\f x\right\}; \\[4pt]
\F_4:\; &T^0(x,y)=\frac{1}{2n}t^{0*}(\xi)\left\{\eta(y)\f x-\eta(x)\f y\right\};\\[4pt]
\F_5:\; &T^0(x,y)=\frac{1}{2n}t^0(\xi)\left\{\eta(y)\f^2 x-\eta(x)\f^2 y\right\};\\[4pt]
\F_{11}:\; &T^0(x,y)=\left\{\hat{t^0}(x)\eta(y)-\hat{t^0}(y)\eta(x)\right\}\xi.\\[4pt]
\end{array}
\]
Moreover, it is given the following correspondence between the
classes $\F_i$ of $M$ and the classes $\T_{jk}$ of $T^0$:
\begin{equation}\label{MFiT0Tjk}
\begin{array}{l}
M\in\F_1\; \Leftrightarrow \; T^0\in\T_{13},\; t^0\neq 0; \\[4pt]
M\in\F_4\; \Leftrightarrow \; T^0\in\T_{31},\; t^0= 0,\;
t^{0*}\neq 0;
\\[4pt]
M\in\F_5\; \Leftrightarrow \; T^0\in\T_{31},\; t^0\neq 0,\;
t^{0*}= 0;
\\[4pt]
M\in\F_{11}\; \Leftrightarrow \; T^0\in\T_{41}.
\end{array}
\end{equation}

\section{Natural Connections in the Main Classes}

Let $\M$ be an almost contact B-metric manifold  belonging to the
main classes $\F_1$, $\F_4$, $\F_5$ and $\F_{11}$. The goal in the
present section is a determination of the general form of the
torsion $T$ of a natural connection $D$ on $\M$.

Bearing in mind the form of the torsion of the $\f$-canonical
connection $\eqref{T0}$ and characteristic conditions \eqref{Fi}
of the mentioned classes, we expect that the torsion $T$ is
expressed by the components from \eqref{g} and \eqref{tg} of the
metrics $g$ and $\tilde{g}$, respectively.

Let the torsion $T$ of a linear connection $D$ have the following
form
\begin{equation}\label{T-nat-g}
\begin{split}
T(x,y,z)&
=g(\f y,\f z)\vartheta_1(x)+g(y,\f z)\vartheta_2(x)+\eta(y)\eta(z)\vartheta_3(x)\\[4pt]
& -g(\f x,\f z)\vartheta_1(y)-g(x,\f
z)\vartheta_2(y)-\eta(x)\eta(z)\vartheta_3(y),
\end{split}
\end{equation}
where, for $\lm_i\in\R$ ($i=1,2,\dots,18$), there are denoted the
following 1-forms:
\begin{equation}\label{Ai-nat-g}
\begin{split}
\vartheta_1(x)& =\lm_1\ta(\f^2x)+\lm_3\ta(\xi)\eta(x)+\lm_5\om(x)
\\[4pt]
&+\lm_2\ta^*(\f^2x)+\lm_4\ta^*(\xi)\eta(x)+\lm_6\om(\f x),
\\[4pt]
\vartheta_2(x)&
=\lm_7\ta(\f^2x)+\lm_9\ta(\xi)\eta(x)+\lm_{11}\om(x)
\\[4pt]
&+\lm_8\ta^*(\f^2x)+\lm_{10}\ta^*(\xi)\eta(x)+\lm_{12}\om(\f x),
\\[4pt]
\vartheta_3(x)&
=\lm_{13}\ta(\f^2x)+\lm_{15}\ta(\xi)\eta(x)+\lm_{17}\om(x)
\\[4pt]
&+\lm_{14}\ta^*(\f^2x)+\lm_{16}\ta^*(\xi)\eta(x)+\lm_{18}\om(\f
x).
\end{split}
\end{equation}

We set the condition $D$ to be a natural connection on
$(M,\f,\xi,\allowbreak\eta,g)$. According to
\propref{prop-nat-con} and \eqref{Hay}, we obtain the following
relation
\begin{equation}\label{F=T}
\begin{split}
F(x,y,z)=\frac{1}{2}\bigl\{&T(x,y,\f z)-T(y,\f z,x)+T(\f
z,x,y)\\[4pt]
&-T(x,\f y,z)+T(\f y,z,x)-T(z,x,\f y)\bigr\}.
\end{split}
\end{equation}
Then, applying \eqref{T-nat-g} and \eqref{Ai-nat-g} to \eqref{F=T}
and comparing with \eqref{F14511}, we obtain the following
conditions for the parameters:
\[
\lm_1+\lm_8=\lm_5-\lm_{12}=\lm_6+\lm_{11}=0,\quad
\lm_2-\lm_7=-\lm_4=\lm_9=\frac{1}{2n},
\]
\[
\lm_3=\lm_{10}=\lm_{13}=\lm_{14}=\lm_{15}=\lm_{16}=\lm_{17}=0,\quad
\lm_{18}=-1.
\]
In consequence, the 1-forms from \eqref{Ai-nat-g} take the form
\begin{equation}\label{Ai-nat-g=}
\begin{split}
\vartheta_1(x)&
=\lm_1\ta(\f^2x)+\lm_2\ta^*(\f^2x)-\frac{1}{2n}\ta^*(\xi)\eta(x)
\\[4pt]
&+\lm_5\om(x)+\lm_6\om(\f x),
\\[4pt]
\vartheta_2(x)&
=\left(\lm_2-\frac{1}{2n}\right)\ta(\f^2x)-\lm_1\ta^*(\f^2x)+\frac{1}{2n}\ta(\xi)\eta(x)
\\[4pt]
&-\lm_{6}\om(x)+\lm_{5}\om(\f x),
\\[4pt]
\vartheta_3(x)& =-\om(\f x).
\\[4pt]
\end{split}
\end{equation}

Further, we use the following tensors derived by the structural
tensors of $\M$:
\[
\begin{split}
\pi_1(x,y)z&=\bigl\{g(y,z)x\bigr\}_{[x\leftrightarrow y]},\qquad
\pi_2(x,y)z=\bigl\{g(y,\f z)\f x\bigr\}_{[x\leftrightarrow y]},
\\[4pt]
\pi_3(x,y)z&=-\bigl\{g(y,z)\f x+g(y, \f z)
x\bigr\}_{[x\leftrightarrow y]},
\\[4pt]
\pi_4(x,y)z&=\bigl\{\eta(y)\eta(z)x +g(y,
z)\eta(x)\xi\bigr\}_{[x\leftrightarrow y]},\\[4pt]
\pi_5(x,y)z&=\bigl\{\eta(y)\eta(z)\f x +g(y,\f
z)\eta(x)\xi\bigr\}_{[x\leftrightarrow y]},
\end{split}
\]
as well as the corresponding tensors
$\pi_i(x,y,z,w)=g\left(\pi_i(x,y)z,w\right)$, $i=1,\dots,5$, of
type (0,4). Let us note that the latter tensors are curvature-like
tensors, i.e. they have the properties of the curvature tensor
$R=\left[\n,\n\right]-\n_{[,]}$ of type (0,4).

In \eqref{Ai-nat-g=}, we rename the parameters as follows
$\al_1=\lm_1$, $\al_2=\lm_2$, $\al_3=\lm_5$, $\al_4=\lm_6$. Then,
using \eqref{T-nat-g} and \eqref{Ai-nat-g=}, we establish the
truthfulness of the following
\begin{thm}\label{thm:nat-con}
The torsion of any natural connection on an almost contact
B-metric manifold in $\F_1\oplus\F_4\oplus\F_5\oplus\F_{11}$ is
expressed by
\begin{equation}\label{T-nat}
\begin{split}
    T(x,y,z)&=(\pi_3+\pi_5)(x,y,z,q)\\[4pt]
    &+\frac{1}{2n}(\pi_2+\pi_4)(x,y,z,a^*)+\frac{1}{2n}\pi_5(x,y,z,a)-\pi_5(x,y,z,\hat{a}),
\end{split}
\end{equation}
where $q=\al_1\f^2 a^*-\al_2\f^2 a-\al_3\f\hat{a}+\al_4\hat{a}$
$(\al_1,\al_2,\al_3,\al_4 \in \R)$ and $\ta=g(\cdot,a)$,
$\ta^*=g(\cdot,a^*)$, $\om=g(\cdot,\hat{a})$.
\end{thm}

Actually, the tensors in \eqref{T-nat} form a 4-parametric family.

By direct check, we establish that equality \eqref{T-nat} implies
the following
\begin{prop}\label{prop-sT}
The torsion of any natural connection on an almost contact
B-metric manifold in $\F_1\oplus\F_4\oplus\F_5\oplus\F_{11}$ has
the property
\[
\sx T(x,y,z)=0.
\]
\end{prop}

As a corollary, using \eqref{Hay} and \propref{prop-sT}, we obtain
immediately $Q(x,y,z)=T(z,y,x)$, i.e. we have the following
\begin{prop}\label{prop-QT}
Any natural connection $D$ with torsion $T$ on an almost contact
B-metric manifold from $\F_1\oplus\F_4\oplus\F_5\oplus\F_{11}$ is
determined by $T$ and the Levi-Civita connection $\n$ as follows:
\[
g\left(D_xy,z)\right)=g\left(\n_xy,z)\right)+T(z,y,x).
\]
\end{prop}

\section{Special Natural Connections in the Main Classes}

Bearing in mind that the $\f$-canonical connection (resp., the
$\f$B-connec\-tion) is a natural connection on $\M$, then we have
to determine its corresponding parameters $\al_i$ ($i=1,2,3,4$) in
\eqref{T-nat}.

\begin{thm}\label{thm:can-con}
The natural connection with torsion $T$ from the family
\eqref{T-nat} is the $\f$-canonical connection on a
$(2n+1)$-dimensional almost contact B-metric manifold from
$\F_1\oplus\F_4\oplus\F_5\oplus\F_{11}$ if and only if
$\al_1=\al_3=\al_4=0$, $\al_2=\frac{1}{4n}$. Then its torsion has
the form
\begin{equation}\label{T0-can}
\begin{split}
    T^0(x,y,z)&=\frac{1}{4n}(\pi_1+\pi_2+\pi_4)(x,y,z,a^*)\\[4pt]
    &+\frac{1}{2n}\pi_5(x,y,z,a)-\pi_5(x,y,z,\hat{a}).
\end{split}
\end{equation}
\end{thm}
\begin{proof}
The statement follows immediately from \eqref{T-nat}, \eqref{T0}
and \eqref{F14511}.
\end{proof}

By virtue of \eqref{T0-can} and \eqref{Tij}, we obtain the
following

\begin{prop}\label{prop:T-can-Tij}
The $\f$-canonical connection on an almost contact B-metric
manifold from $\F_1\oplus\F_4\oplus\F_5\oplus\F_{11}$ has a
torsion belonging to the class
$\T_{13}\oplus\T_{31}\oplus\T_{41}$.
\end{prop}

The latter fact confirms the results in \eqref{MFiT0Tjk}.

\thmref{thm:nat-con} gives an one-to-one map of the set of the
natural connections onto the set of the quadruples
$(\al_1,\al_2,\al_3,\al_4)\in\R^4$. According to
\thmref{thm:can-con}, the $\f$-canonical connection is determined
by $\left(0,\frac{1}{4n},0,0\right)$. The interesting case is the
 connection from the family \eqref{T-nat} for $\left(0,0,0,0\right)$. This
natural connection we call the \emph{standard connection} on
$\M\in\F_1\oplus\F_4\oplus\F_5\oplus\F_{11}$.

Then, according to \eqref{T-nat}, its torsion has the form
\begin{equation}\label{T-st}
\begin{split}
    T'(x,y,z)&=\frac{1}{2n}(\pi_2+\pi_4)(x,y,z,a^*)\\[4pt]
    &+\frac{1}{2n}\pi_5(x,y,z,a)-\pi_5(x,y,z,\hat{a}).
\end{split}
\end{equation}
Using \eqref{T-st} and \eqref{Tij}, we get the following

\begin{prop}\label{prop:T-st-Tij}
The standard connection on an almost contact B-metric manifold
from $\F_1\oplus\F_4\oplus\F_5\oplus\F_{11}$ has a torsion
belonging  to the class
$\T_{11}\oplus\T_{13}\oplus\T_{31}\oplus\T_{41}$.
\end{prop}

There exist a natural connection $\n''$ with torsion $T''$ from
the family \eqref{T-nat} of a remarkable role. This connection is
determined by the condition the $\f$-canonical connection to be
the average connection of $\n'$ and $\n''$. Therefore we have
\begin{equation}\label{T-co-st}
\begin{split}
    T''(x,y,z)&=\frac{1}{2n}(\pi_3+\pi_5)(x,y,z,a)+\frac{1}{2n}(\pi_2+\pi_4)(x,y,z,a^*)\\[4pt]
    &+\frac{1}{2n}\pi_5(x,y,z,a)-\pi_5(x,y,z,\hat{a}).
\end{split}
\end{equation}
Bearing in mind \eqref{T-co-st} and \eqref{Tij}, we get the
following

\begin{prop}\label{prop:T-co-st-Tij}
The natural connection $\n''$ on an almost contact B-metric
manifold from $\F_1\oplus\F_4\oplus\F_5\oplus\F_{11}$ has a
torsion belonging to the class
$\T_{11}\oplus\T_{13}\oplus\T_{31}\oplus\T_{41}$.
\end{prop}

\bigskip

\small{ \noindent
\textsl{M. Manev, M. Ivanova\\
Department of Algebra and Geometry\\
Faculty of Mathematics and Informatics\\
University of Plovdiv\\
236 Bulgaria Blvd\\
4003 Plovdiv, Bulgaria}
\\
\texttt{mmanev@uni-plovdiv.bg, mirkaiv@uni-plovdiv.bg} }

\end{document}